\documentclass[12pt]{amsart}
\usepackage{amsmath}

\def\ol{\overline}
\def\e{\epsilon}
\def\hb{holomorphic bisectional curvature}
\def\lf{\left}
\def\ri{\right}

\def\wt{\widetilde}

\def\p{\partial}

\newcommand\ce{{\mathbb C}}
\newcommand\C{{\mathbb C}}

\def\ii{\sqrt{-1}}
\def\jbar{{\bar\jmath}}

\def\ttg{\tilde{g}}
\def\ttR{\tilde{R}}

\def\tg{\tilde{g}_{i\bar\jmath}}

\def\K{K\"ahler }
\def\KR{K\"ahler-Ricci }
\def\KRF{K\"ahler-Ricci flow }
\def\KRS{K\"ahler-Ricci soliton }
\def\KRS{K\"ahler-Ricci soliton }

\def\be{\begin{equation}}
\def\ee{\end{equation}}
\def\ol{\overline}
\def\lf{\left}
\def\ri{\right}

\def\ttg{\wt{g}}
\def\ttR{\wt  R}
\def\tM{\wt M}
\def\tN{\wt N}
\def\tth{\wt h}
\def\tx{\tilde x}
\def\ty{\tilde y}

\def\hg{\hat g}
\def\e{\epsilon}

\def\ijb{{i\jbar}}

\def\Ric{\text{\rm Ric}}

\def\wt{\widetilde}

\def\p{\partial}
\def\tM{\tilde{M}}
\def\tN{\tilde{N}}
\def\ttg{\tilde{g}}
\def\tth{\tilde{h}}
\def\tx{\tilde{x}}
\def\ty{\tilde{y}}
\def\e{\epsilon}

\def\C{\Bbb C}
\def\cn{\Bbb C^n}
\def\wh{\widehat}

\def\wt{\widetilde}

\def\p{\partial}

\def\C{\Bbb C}
\def\Rm{\text{\rm Rm}}
\def\ii{\sqrt{-1}}
\def\KRF{K\"ahler-Ricci flow }
\def\ttg{\wt{g}}
\def\ttR{\wt  R}
\def\inj{\text{\rm inj}}
\newtheorem{thm}{Theorem}[section]

\newtheorem{lem}{Lemma}[section]
\newtheorem{prop}{Proposition}[section]
\newtheorem{cor}{Corollary}[section]
\theoremstyle{definition}

\theoremstyle{remark}
\newtheorem{rem}{Remark}[section]
\numberwithin{equation}{section}
\title{On the simply connectedness of non-negatively curved K\"ahler manifolds and applications }

\author{Albert Chau$^1$}
\address{Department of Mathematics,
The University of British Columbia,
Room 121, 1984 Mathematics Road,
Vancouver, B.C., Canada V6T 1Z2}
\email{chau@math.ubc.ca}

\author{Luen-Fai Tam$^2$}

\thanks{$^1$Research
partially supported by NSERC grant no. \#327637-06}
\thanks{$^2$Research
partially supported by Earmarked Grant of Hong Kong \#CUHK403005}

\address{The Institute of Mathematical Sciences and Department of
 Mathematics, The Chinese University of Hong Kong,
Shatin, Hong Kong, China.} \email{lftam@math.cuhk.edu.hk}

\begin{document}

 \maketitle

 \begin{abstract} We study complete noncompact long time solutions $(M, g(t))$ to the K\"ahler-Ricci
  flow with uniformly bounded nonnegative holomorphic bisectional curvature.
    We will show that when the Ricci curvature is positive and  uniformly
    pinched, i.e.  $ R_\ijb \ge cRg_\ijb$ at $(p,t)$  for
all $t$ for some $c>0$, then there always exists a local gradient
\KR soliton limit around $p$ after possibly rescaling $g(t)$ along
some sequence $t_i \to \infty$.  We will show as an immediate
corollary that the injectivity radius of $g(t)$ along $t_i$ is
uniformly bounded from below along $t_i$, and thus $M$ must in
fact be simply connected.  Additional results concerning the
uniformization of $M$ and fixed points of the holomorphic isometry
group will also be established.  We will then consider removing
the condition of positive Ricci for $(M, g(t))$.  Combining our
results with Cao's splitting for \KRF \cite{Cao04} and techniques
of Ni-Tam \cite{NiTam03}, we show that when the positive
eigenvalues of the Ricci curvature are uniformly pinched at some
point $p \in M$, then $M$ has a special holomorphic fiber bundle
structure.  We will treat as special cases, complete  \K manifolds
with non-negative \hb\ and average quadratic curvature decay
aswell as the case of steady gradient \KR solitons.
 \end{abstract}

 \markboth{Albert Chau and Luen-Fai Tam} {On the simply connectedness of K\"ahler manifolds }

\section{introduction}
In this paper we study the class of complete noncompact K\"ahler
manifolds $(M^n,g)$ of complex dimension $n$ with bounded and
nonnegative \hb.   Let $R$ be the scalar curvature and let
\begin{equation}\label{averagedecay-e1}
   k(r,x)=\frac{1}{V_x(r)}\int_{B_x(r)}R
\end{equation}
be the average of the scalar curvature over the geodesic ball with
radius $r$ and center  at $x$. In \cite{ChauTam06-2},
\cite{ChauTam07-1} and \cite{ChauTam07-2} the authors obtained,
among other things, the following uniformization results:

{\it (1) If $k(r,x)\le C/(1+r^2)$ for some $C$ for all $r$ and $x$,
then  the universal cover of $M$ is biholomorphic to $\cn$.

(2) Under a weaker assumption that   $k(r,x)\le C/(1+r)$ for some
$C$ for all $r$ and $x$,   the universal cover of $M$ is
homeomorphic to $\mathbb{R}^{2n}$ and is biholomorphic a
pseudoconvex domain in $\cn$.}

The results support the following conjecture due to Yau
\cite{Yau91}:  {\sl If $(M, g)$ has positive \hb\ then $M$ is
biholomorphic to $\C^n$.}

In this paper we want to discuss the structure of $M$ itself.  In
light of the above results, a reasonable approach is as follows.
First one studies more closely the structure of the universal
cover $\tM$. Then one would like to study properties of the first
fundamental group, and hopefully these together will provide
information about the structure of $M$. As in the previous works,
we will study the \KR flow equation on $(M, g)$:
\begin{equation}\label{krf-e0}
   \frac{\p g_{i\jbar}}{\p t}=-R_{i\jbar}.
\end{equation}
By the works \cite{Shi89} and \cite{Shi97} of Shi, the \KR flow
has a long time solution $g(t)$ with uniformly bounded curvature
under  assumption (2) (and hence (1)) above. By the splitting
result of Cao \cite{Cao04}, $(\tM, \wt g(t))$, the pull back of
$g(t)$ to the universal cover $\tM$, splits for all $t>0$ as
$\tM=\C^k\times \tN$ and $\wt g(t)=g_e\times h(t)$ where $g_e$ is
the standard metric on $\C^k$ and $h(t)$ has nonnegative \hb\  and
positive Ricci curvature. Hence we will begin our studies with
long time solutions $(M,g(t))$ to \KRF where $g(t)$ has uniformly
bounded nonnegative \hb\  and positive Ricci curvature.  We will
prove that:

{\it  If at some
point $p\in M$, the eigenvalues of $\Ric(p,t)$ are uniformly
pinched  (i.e. the smallest eigenvalue of $\Ric(p,t)$ is at least
$C$ times the largest eigenvalue for some $C>0$ independent of
$t$), then $M$ is simply connected and is in fact biholomorphic to
a pseudoconvex domain in $\C^n$}.

See Theorems \ref{positiveR}, \ref{fastdecayR} and
\ref{slowdecayR}.  A particular case is when $(M, g)$ has average
quadratic curvature decay and positive Ricci in which case it is
shown in Corollary 2.1 that $M$ is biholomrophic to $\C^n$.  This
generalizes previous results in \cite{ChauTam07-1} where the same
result, in particular Corollary 2.1, is proved assuming in
addition either that \eqref{krf-e0} has an {\it eternal solution}
(i.e. $g(t)$ exists for $-\infty<t<\infty$) or that $tR$ is
uniformly bounded in spacetime.  Note that either conditions will
imply that eigenvalues of $\Ric(p,t)$ are uniformly pinched, see
\cite{ChauTam07-1} for example.

 More   generally,  if we only assume the {\it positive} eigenvalues of
$\Ric(p,t)$ are uniformly pinched, then using the above result on
simple connectedness, we prove that:

{\it  $M$ is a holomorphic and Riemannian fiber bundle over
$\C^k/\Gamma$ where $\Gamma$ is a discrete subgroup of the
holomorphic isometry group of $\C^k$, with fiber either $\C^{n-k}$
or a pseudoconvex domain in $\C^{n-k}$}.

See Theorem \ref{fiberbundle}.  In particular, the  fiber is
$\C^{n-k}$  when $(M, g)$ has average quadratic curvature decay,
which we show in Corollary 3.1. Similar results under different
conditions have been obtained by Takayama \cite{Takayama}, Ni-Tam
and Zheng \cite{NiTam03}.

Our results are related to the second part of Yau's conjecture in
\cite{Yau91}: {\sl If the \hb\ is only non-negative then $M$ is
biholomorphic to a complex vector bundle over a compact Hermitian
symmetric space. }  In fact, we prove that if $n-k=1$ in the above
situation, then $M$ is actually a line bundle over $\C^k/\Gamma$.
In order to prove this, we have to study more carefully about the
isometry group of a complete noncompact K\"ahler manifold $(M,g)$
with nonnegative \hb\ and positive Ricci curvature. We prove that
if there is a long time solution $g(t)$ to \KRF with uniformly
bounded non-negative \hb\ and positive uniformly pinched Ricci
curvature, then any finite set of isometries of $M$ has a fixed
point.  Under the stronger assumption that \eqref{krf-e0} has an
eternal solution or $tR(p,t)$ is uniformly bounded for all time,
then one can prove the stronger result that there is a fixed point
for all isometries of $(M,g)$.  See Theorems \ref{fixedpoint-t1}
and \ref{fixedpoint-t2}. One might want to compare this with a
theorem of Cartan which says that a compact subgroup of isometries
of a Cartan-Hadamard manifold has a fixed point.

A special case of the above is when $g(t)$ is actually  a gradient
\KR soliton.  Namely,  there is a real valued function $f$ with
$f_{ij}=0$ such that $R_\ijb=f_\ijb$ (steady type) or after
rescaling $-R_\ijb+g_\ijb=f_\ijb$ (expanding type).  In the case
of expanding type, it is known that $M$ is biholomorphic to $\cn$
by \cite{ChauTam05}.  For the case of steady type, under an
additional assumption that the scalar curvature attains its
maximum at some point,  we prove that in most cases, $M$ is a
holomorphic vector bundle or $\C^k/\Gamma$. More precisely, we
prove that this is the case    if the scalar curvature $R$ is
maximal at $p\in M$ say, and the positive eigenvalues
$\lambda_1,\dots,\lambda_l$ of $\Ric$ at $p$ satisfy the
Diophantine condition
$$
\sum_{i=1}^l m_i\lambda_i\neq \lambda_j
$$
for all $j$ and for all nonnegative integers $m_i$  such that
$\sum_{i=1}^lm_i\ge2$. See Theorem \ref{steady-s3t2}. The proof
relies on the result of Bryant \cite{Bryant07} on the existence of
Poincar\'e coordinates. This is related to a result of Yang
\cite{Bo08}: Suppose $M$ is a noncompact K\"ahler-Ricci soliton
with nonnegative Ricci curvature such that its scalar curvature
attains a positive maximum at a compact complex submanifold $K$
with codimension 1 and the Ricci curvature is positive away from
$K$, then $M$ is a holomorphic line bundle over $K$. proves that

We now describe the organization of the paper.  In \S2 we study
long time solutions to \KRF with bounded non-negative \hb\ and
positive Ricci curvature.  We will prove the existence of gradient
\KR soliton limits and consequences of this concerning the
injectivity radius along the flow and the simple connectedness of
$M$. We will also prove results on the existence of fixed point of
isometries.   In \S3  will prove  the fiber bundle structure and
line bundle of $M$. Finally, in \S4 we prove vector bundle
structure on gradient \KR solitons.  We will also prove a result
on the volume growth of expanding gradient Ricci solitons in this
section which is related to the work of Ni \cite{Ni05-1} \footnote{This result (Proposition 4.1) is proved independently in \cite{CN}.}.

 \section{long time solutions and local limits}

In this section we study limits of long time solutions of the \KRF
\eqref{krf-e1} on complete noncompact K\"ahler manifolds with
uniformly bounded nonnegative \hb\ and positive Ricci curvature.
We will investigate when such a solution subconverges, after
rescaling, to a gradient \KR soliton limit, and we will derive a
number of consequences from this.  In Corollary 2.1 we treat the
particular case of a complete non-compact \K manifold having
bounded non-negative \hb\  and positive Ricci curvature which
decays quadratically in the average sense. Our study of the
limiting behavior of \KRF in this section can be viewed as an
extension of that in \cite{ChauTam06-2} and \cite{ChauTam07-1}. We
begin with the following

{\sc Basic Assumption 1:} {\it Let $M^n$ be a noncompact complex
manifold of complex dimension $n$.  Let $g(t)$ be a complete
solution of the \KRF
\begin{equation}\label{krf-e1}
   \frac{\p g_{i\jbar}}{\p t}=-R_{i\jbar}
\end{equation}
on $M\times[0,\infty)$ with uniformly bounded nonnegative
holomorphic bisectional curvature in spacetime and with positive
Ricci curvature.}

Let $R(x,t)$ be the scalar curvature of $g(t)$ at $x$. Let $p\in
M$ be a fixed point. Then we have the following possibilities: (1)
$R(p,t)\ge c$ for some constant $c>$ for all $t$; or (2)
$\inf_{t\in [0,\infty)}R(p,t)=0$. In the second case, since
$tR(p,t)$ is nondecreasing in time by a Li-Yau-Hamilton (LYH) type
differential inequality of Cao   \cite{Cao92}, we have two
subcases: (2a) $tR(p,t)\le c$ for some constant $c>0$; and (2b)
$\lim_{t\to\infty}tR(p,t)=\infty$. We address these cases
separately in Theorems \ref{positiveR}, \ref{fastdecayR} and
\ref{slowdecayR} below.

First, we have the following lemma.

\begin{lem}\label{harnack-L1} Let $(M,g(t))$ be as in the {\sc Basic
Assumption 1}. Let $p\in M$ be a fixed point. There exists $C>0$,
such that for any $\frac14\le \alpha_0\le \alpha_1\le \frac12$,
 any  $\rho>0$,   and any $T>0$ we have
\begin{equation}\label{harnack-e1}
  \frac{ R(x,4\alpha_1T)}{  R(y,3T)}\le C\exp( \frac {\rho^2}4)
\end{equation}
for   $x, y \in B_{4\alpha_0T}(p,\rho\sqrt T)$.
\end{lem}
\begin{proof}  For $t>0$, fixed. Define
$$
\hg(s)=\frac{e^s}tg(t(1-e^{-s})), \ 0\le s<\infty.
$$
Then we have
$$
\frac{\p \hg}{\p s}=-\wh\Ric+\hg.
$$
on $M\times [0, \infty)$.  By the (LYH) type differential inequality in
\cite{Cao92}, for $x, y\in M$ and $0<s_1<s_2<\infty$ we have
\begin{equation}\label{11-e1}
  \frac{\wh R(x,s_1)}{\wh R(y,s_2)}\le \frac{e^{s_2}-1}{e^{s_1}-1} \exp(\frac14\Delta)
\end{equation}
where
$$
\Delta=\Delta(x,y;s_1,s_2)=\inf_\gamma\int_{s_1}^{s_2}|\gamma'(s)|^2_{\hg(s)}ds
$$
and $\gamma$ is any curve from $x$ to $y$. Here $\wh R$ is the
scalar curvature and $\wh \Ric$ is the Ricci curvature of $\hg$

Let $\alpha_1=1-e^{-s_1}$ and $\alpha_2=1-e^{-s_2}$. Then $\wh R(s_i)=(1-\alpha_i)R(\alpha_it)$ for
$i=1, 2$, and we have $0<\alpha_1<\alpha_2$.  Let $ \alpha_0>0$ be fixed. Then $\hg(s)\le
\frac{e^s}tg(\alpha_0t)$ for $\alpha_0\le1-e^{-s}$ because $g(t)$
is decreasing in $t$. Now for $x,y\in B_{\alpha_0t}(p,\rho\sqrt t)$,
let $\gamma$ be a minimal geodesic from $x$ to $y$ with respect to
$g(\alpha_0t)$. If $\frac14\le \alpha_0\le \alpha_1\le \frac12$
and $\alpha_2=\frac34$, then
$$
\Delta(x,y;s_1,s_2)\le \frac{d^2}{t(s_2-s_1)^2} \int_{s_1}^{s_2}e^sds\le
C_1\exp( \frac {\rho^2}4)
$$
for some constant $C_1$ independent of $x, y, t, \rho, \alpha_0,
\alpha_1$, where $d$ is the distance between $x , y$ with respect
to $g(\alpha_0 t)$. Hence
$$
\frac{ R(x,\alpha_1t)}{  R(y,\frac34t)}\le C_2\exp( \frac
{\rho^2}4)$$ for $x, y \in B_{\alpha_0t}(p,\rho\sqrt t)$ and
$\frac14\le\alpha_0\le \alpha_1\le \frac12$. Here $C_2$ is
independent of $x, y, t, \rho, \alpha_0,\alpha_1$.

Replacing $\frac14 t$ by $T$, we have
\begin{equation}\label{11-e2}
  \frac{ R(x,4\alpha_1T)}{  R(y,3T)}\le C_3\exp( \frac {\rho^2}4)
\end{equation}
for   $x, y \in B_{4\alpha_0T}(p,\rho\sqrt T)$ for some constant
$C_2$ independent of $x, y, t, \rho, \alpha_0,\alpha_1$. From this
the lemma follows.
\end{proof}

We will also need the following result which is basically from
\cite{TianYauY90,ChauTam06-2} (see also
\cite{ChauTam07-1}).
\begin{lem} \label{normalcoordinates-L2} Let $(M,g)$ be a complete K\"ahler
manifold. Let $p\in M$ and $r_0>0$. Suppose
  $|Rm(g)|+|\nabla Rm(g)|\le c$ for
some $c$ on $ B(p,r_0)$.
  Then there exist   $r>0$ and a constant $C>0$ depending only on $r_0$, $c$ and $n$, and  a   holomorphic
  map
\begin{equation}\label{s1e2}
\Phi: D(r)\to M
\end{equation} with the following
  properties:
  \begin{enumerate}
\item [(i)] $\Phi$ is a local biholomorphism from $D(r)
    \subset \C^n$ onto its image,

\item[(ii)] $\Phi(0)=p$,

\item[(iii)] $\Phi^*(g)(0)=g_e$,

\item[(iv)] $\frac{1}{C}g_{e}\leq\Phi^*(g_i) \leq  Cg_{e}$
    in $D(r)$,
\end{enumerate}
 where $g_e$ is the standard metric on $\cn$.
\end{lem}

We may now state our first

\begin{thm}\label{positiveR} Suppose that at
some $p \in M$ we have  $\Ric(p,t)\ge cg(p,t)$   for all $t$ for some
$c>0$.  Then the following are true.
\begin{itemize}

\item [(i)] The injectivity radius of $g(t)$ at $p$ satisfies
    $\inj_t(p)\ge a$ for some $a>0$ for all $t$.
\item [(ii)] $M^n$ is homeomorphic to the Euclidean space and is biholomorphic to a pseudoconvex
domain in $\cn$.
\item [(iii)] For any $t_i\to \infty$ we can find a
        subsequence $t_{i_k}$ such that $(M,g(t+t_{i_k}),p)$
        converges to a complete  eternal solution
        $(M_\infty, g_\infty(t), o)$ of \eqref{krf-e1} with positive Ricci curvature and
        bounded non-negative \hb, and
        $M_\infty$ is biholomorphic to $\cn$.
\item[(iv)] For any  $t_i\to\infty$, there exist a subsequence
    $t_{i_k}$ and  $s_{i_k}\ge t_{i_k}$ such that $(M, g(s_{i_k}), p)$
    converges to a steady gradient K\"ahler-Ricci soliton
    with positive Ricci curvature and bounded nonnegative \hb \ which is biholomprhic to $\cn$.
\end{itemize}
\end{thm}
\begin{proof} Assume (i) is true. Since the curvature, and its covariant derivatives,
 of $g(t)$ are uniformly bounded by
\cite{Shi97} for $t\ge1$,  for each $t\ge 1$  there is a
biholomorphism $\Phi_t$ from $D(r)\to M$
satisfying the conditions in Lemma
\ref{normalcoordinates-L2} with $r$ and $C$ being
independent of $t$. Since $\Ric>0$ for all $t$, it is easy
to see that $\Phi_t(D(r))$ exhausts $M$ as $t\to \infty$ by the eigenvalue inequality for $Rc(t)$ in \cite{ChauTam07-2} (Theorem 6.1). Hence $M$ is homeomorphic to the
Euclidean space by \cite{Brown61}. This proves the first
part of (ii). The second part of (ii) then follows from Theorem 1.2 in
\cite{ChauTam07-2}.

(iii) follows from (i), the standard result on compactness of
solutions of Ricci flow \cite{Hamilton95-1} and  Theorem 1.4 in
\cite{ChauTam07-1}.

(iv) will follows from (i) and the content in the proof of (i).

It remains to prove (i). Suppose there exist $t_i\to\infty$
such that $\lim_{i\to \infty}\inj_{t_i}(p)=0$. Since the
curvature are uniformly bounded in spacetime, for any
positive integer $k\ge1$ there is a positive constant $C_1$
which is independent of $i$, but possibly depending on $k$, such
that
$$
\frac1{C_1}g(t)\le g(t_i)\le C_1g(t)
$$
 and for any $t\in [t_i,t_i+k]$. By volume comparison and
 \cite{CheegerGromovTaylor} on the injectivity radius bound
 we have:
 $$
 \lim_{i\to\infty}\sup_{t\in [t_i,t_i+k]}\inj_{t}(p)=0
 $$
for any $k$.  On the other hand, since the scalar curvature of $g(t)$ are
uniformly bounded, for any $i$, there exists $s_i\in
 [t_i,t_i+k]$ such that the scalar curvature $\lf|\frac{\p}{\p t}R(p,t)|_{t=s_i}\ri|\le
 C_2/k$ where $C_2=2\sup_t R(p,t)$.  By passing to a
 subsequence still denoted by $t_i$,  we can
 find $k_i \to \infty$ and $s_i\in [t_i, t_i+k_i]$ such that
\begin{equation}\label{positiveR-e1}
   \lim_{i\to\infty}\inj_{s_i}(p)=0,\
   \lim_{i\to\infty}\frac{\p}{\p t} R(p,t)|_{t=s_i}=0.
\end{equation}
By Lemma \ref{normalcoordinates-L2}, there is $r>0$ and
$C_3>0$ independent of $i$ such that for each $i$, there
exists a  a holomorphic map $\Phi_i: D(r)\to M$
corresponding to the metric $g(s_i)$ satisfying the
conditions in the lemma with $C=C_3$ in property (iv) in
the lemma.
  Now by considering the
sequence $(D(r), g_i (s)=\Phi_{i}^* g(s_i +s))$, we may
then obtain a local subsequence limit solution $(D(r),
h(z,t))$ to (\ref{krf-e1}) on an eternal time interval and
satisfying $\frac{\p}{\p t} R^h (0, 0)=0$ where $R^h$ is
the scalar curvature of $h$. By Theorem 2.1 in
\cite{ChauTam07-1} we find that $(D(r), h(0))$ is a local
steady gradient \KR soliton with positive Ricci curvature
centered at the origin. Namely, there is a real valued
function $f$ on $D(r)$ such that $R^h_{\ijb}=f_\ijb$ and
$f_{ij}=0$ in $D(r)$, where $R^h_{\ijb}$ is the Ricci
curvature of $h(0)$. Thus $Z^i=f^i$ is a holomorphic vector
field, and it can be shown that the gradient of $f$ is zero
at the origin (see \cite{ChauTam07-1} for example), which
is an isolated zero because $f_\ijb=R^h_\ijb>0$. In fact by
 its construction in
  \cite[\S2]{ChauTam07-1} and \cite{Cao97}, there exists $0<r_1<r$ such that
   $Z$ is the unique solution of the equation
\begin{equation}\label{vectorfield-e2}
R_{,i}+R_{i\bar j}Z^{\bar j}=0
\end{equation}
in $D(r_1)$ relative to the metric $h(0)$.

Note that (\ref{vectorfield-e2}) has a unique solution $Z
\in T^{(0,1)}$ on any \K manifold having positive Ricci
curvature.  By our assumptions on the positivity of Ricci
curvature, we may then let $W(i) \in T^{(0, 1)} M$ be the
unique solutions to  (\ref{vectorfield-e2}) on $(M,
g(s_i))$ for any $k$.  In particular, for any $i$ as
$\Phi_{i}$ is a local biholomorphism,
$V(i)=\lf(\Phi_{i}\ri)^*(W(i))$ is a smooth $(0,1)$ vector
field in $D(r)$.  By the definition of $g(s_i)$,
$V(i)$ is just the unique solution of (\ref{vectorfield-e2}) in
$D(r_1)$ relative to the metric
$\lf(\Phi_{i}\ri)^*(g(s_i))$. By   uniqueness and the fact
that $\lf(\Phi_{i}\ri)^*(g(s_i)) \to h$, we conclude that:
\vspace{12pt}

{\it $V(i)$ converges to $Z$ uniformly on compact sets in
$D(r_1)$ in the $C^\infty$ topology.}
\vspace{12pt}

It is easy to see that the integral curves of $-Z$ in
$D(r)$ will converge to the origin. Since $V(i)$ converge
to $Z$, one expects the integral curves of $V(i)$ have
similar behaviors if $i$ is large. In fact, In real
coordinates $x_\alpha$ of $D$, the integral curves of $-Z$
are given by the following ODE
\begin{equation}\label{integralcurve-e1}
x_\alpha'=-\lambda_\alpha x_\alpha+F_\alpha(x)
\end{equation}
where $\lambda_\alpha\ge c>0$ are the positive eigenvalues
of $R_\ijb^h$ at the origin,  $|F (x)|=(|x|^2)$ and
$|dF(x)|=O(|x|)$.  On the other hand, for any   $\e>0$ there exists $i_0$ and $0<r_1<r$ such
that the integral curves of $V(i)$ for $i\ge i_0$ is given
by
\begin{equation}\label{integralcurve-e2}
x_\alpha'=-\lambda_\alpha x_\alpha+G_\alpha^i(x)
\end{equation}
  $|G^i-F|+|dG^i-dF|\le \e$ in $D(r_1)$.

Fix $r_1>r_2>0$ and suppose $x(\tau)$  is an integral curve
of $-V(i)$ in $D(r)$. Then
\begin{equation}\label{integralcurve-e2}
   \begin{split}\frac{d}{d\tau}|x|^2&\le
   -2c|x|^2+|G^i(x)|\,|x|\\
   &\le -2c|x|^2+ \e|x|+C_4|x|^3\\
   &\le -\frac32c|x|^2+\e|x|
\end{split}
\end{equation}
where $C_4$ depends only on $F$, provided
$r_2<\frac{c}{2C_3}$ and $|x|\le r_2$. Fix $r_2$ satisfying
these conditions. Then if $\frac {r_2}2\le |x|\le r_2$ we have
\begin{equation}\label{integralcurve-e3}
   \frac{d}{d\tau}|x|^2
\le -\frac32c|x|^2+  \frac{2\e}r|x|^2<0
\end{equation}
if $\e>0$ is small enough. Choose such an $\e$. Hence for
$i\ge i_0$ large enough any integral curve of $V(i)$
starting in $D(r_2)$ will stay inside $D(r_2)$.

Now Let $x(\tau)$ and $y(\tau)$ be two integral curves of
$V(i)$ inside $D(r_2)$. Then
\begin{equation}\label{integralcurve-e4}
   \begin{split}
     \frac{d}{dt}|x-y|^2 & =-2\sum_\alpha\lambda_\alpha
     (x_\alpha-y_\alpha)^2 +
     \sum_\alpha\lf(G^i_\alpha(x)-G^i_\alpha(y)\ri)(x_\alpha-y_\alpha)\\
       &\le -2c|x-y|^2+\sum_{\alpha}||dG^i(z_\alpha)||\,
       (x_\alpha-y_\alpha)^2\\
       &\le -2c|x-y|^2+\e
       |x-y|^2+\sum_{\alpha}||dF(z_\alpha)||\,(x_\alpha-y_\alpha)^2\\
       &\le -c|x-y|^2
   \end{split}
\end{equation}
for some $z_\alpha$ inside $D(r_2)$, provided   $r_2$ and
$\e$ are   sufficiently small, both depending only on the
constant $c$. Hence
\begin{equation}\label{integralcurve-e4}
    |x-y|^2(\tau)\le \exp(-c\tau)|x-y|^2(0)\le
    4r^2\exp(-c\tau).
\end{equation}
In particular, if $x(\tau)$ is an integral curve, and if
$\tau_2>\tau_1>0$, then $y(\tau)=x(\tau_2-\tau_2+\tau)$ is also an
integral curve. Hence $|x(\tau_1)-x(\tau_2)|^2\le
4r^2\exp(-c\tau_1)$.  Hence $x(\tau)$ will converge to a point
$x_0\in \ol{D(r_2)}$, and  by \eqref{integralcurve-e4}, we see
that any set in $D(r_2)$ will contract to $x_0$ along integral
curves. In fact, $x_0\in  \ol{D(\frac {r_2}2)}$ by
\eqref{integralcurve-e3}. Moreover, $x_0$ is  a fixed point of
$V(i)$.

 We now claim that $\Phi_i$ is injective on $D(r_2)$ for $i\ge
 i_0$.
Assume this is not the case and that for some
   $z_1\neq z_2 \in D(r_2)$ we have
   $ \Phi_i (z_1)=\Phi_i (z_2)=q\in M$.
   Let $\gamma_1$ and $\gamma_2$ be two integral curves for
   $V(i)$ starting at $z_1$ and
$z_2$ respectively.  Then by the construction of $V(i)$,
$\Phi_i(\gamma_1)$ and $\Phi_i(\gamma_2)$ are integral
curves for $W(i)$ starting from $q$.  Hence by the
uniqueness of integral curves we must have
$\Phi_i(\gamma_1(\tau))= \Phi_i(\gamma_2(\tau))$ for all
$\tau$.  On the other hand, for all $\tau$, we have
$\gamma_1(\tau)\neq \gamma_2(\tau)$ also by uniqueness of
integral curves. But  $\gamma_1(\tau)$ and $
\gamma_2(\tau)$ both converge to a $x_0 \in \bar D(r_2)$.
It is readily seen that these facts contradict the fact
that $\Phi_i$ is a biholomorphism in some neighborhood of
$x_0$. Thus we have shown that $\Phi_i$ is injective on
$D(r_2)$ for all $i\ge i_0$.

Now It is easy to see that $\Phi_i (D(r_2))$ is contained in
some geodesic ball $B_{s_i}(p,r_3)$ for some $r_3$
independent of $i$.  Thus by the injectivity of $\Phi_i$ shown above, and by the fact that $\Phi_i^*(g(s_i))$ is
uniformly equivalent to the Euclidean metric by Lemma \ref{normalcoordinates-L2}, it is easy to
see that there exists  $b>0$ such that for $i\ge i_0$ large
enough $B_{s_i}(p,r_3) $ has volume bounded below by $b$.
By the result of Cheeger-Gromov-Taylor
\cite{CheegerGromovTaylor},   $\inj_{s_i}(p)$ must be
uniformly bounded away from zero. But this  contradicts our assumption that  $\inj_{s_i}(p) \to 0$.
This completes the proof of the theorem by contradiction.
\end{proof}

Next we consider the case that $tR(p,t)$ is uniformly bounded.

\begin{thm}\label{fastdecayR} Suppose  that at some $p\in M$ we have
$t  R(p,t)\le c$  for all $t$ and some $c>0$. Then the following
are true.
\begin{itemize}

        \item [(i)] The injectivity radius of $g(t)$ at $p$
            satisfies $\inj_t(p)\ge at^{\frac12}$ for some
            $a>0$ for all $t$.
    \item [(ii)] $M^n$ is biholomorphic to $\cn$.
    \item [(iii)] Let $\tg(s)=e^{-s}g(e^s)$.  Then for any
        $s_i\to \infty$ we can find a subsequence $s_{i_k}$
        such that $(M,\tg(s_{i_k}), p)$ converges to an
        expanding gradient K\"ahler-Ricci soliton with
        positive Ricci curvature which is biholmorphic to
        $\cn$.
\end{itemize}
\end{thm}

We will need the following lemma.
\begin{lem}\label{fastdecayR-L1} With the same assumptions as in
Theorem \ref{fastdecayR}, for any $\rho>0$  and integer $k\ge 0$,
there is a constant $C$   such that
$$
||\nabla^k\Rm(x,t)||^2_t\le \frac{C}{t^{2+k}}
$$
for all $x\in B_t(p,\rho\sqrt{t})$ and for all $t>0$, where
$\nabla$ and the norm are taken with respect to $g(t)$.
\end{lem}
\begin{proof} By the assumptions and Lemma
\ref{harnack-L1}, there exists a constant $C_1$ such that for all
$T>0$ and $r>0$,
\begin{equation}\label{fastdecayR-e1}
   \frac{R(x,t)}{R(p,3T)}\le C_1\exp(\frac {r^2}4)
\end{equation}
for all $x\in B_{T}(p, r\sqrt T)$ and $T\le t\le 2T$. Now consider
the metrics $\hg(s)=\frac1T g(T+sT)$. Then $\hg$ satisifies
the \KRF equation with nonnegative \hb. By the assumption that
$tR(p,t)$ is uniformly bounded and \eqref{fastdecayR-e1} and the
fact that $\hg(s)$ has nonnegative \hb, we have
\begin{equation}\label{fastdecayR-e1.5}
||\wh Rm||\le C_2\exp(\frac {r^2}4)
\end{equation}
in $\wh B_0(p, r)\times[0,2]$, where $C_2$ is constant independent
of $r$ and $T$ where $\wh \Rm $, $||\cdot||$ are the curvature
tensor and the norm with respect to $\hg$.  Thus by the local
derivatives estimates of Shi \cite{Shi89}, see also
\cite{ChowLuNi}, for any integer $k\ge 0$, there is a constant
$C_3$ which is independent of $t$ such that
\begin{equation}\label{fastdecayR-e2}
||\wh  \nabla^k\wh \Rm||\le C_3
\end{equation}
in $\hat B_0(p, \frac r2)\times\{1\}$.

On the other hand, by the \KRF equation, we conclude that
$$
C_4\exp(\exp(\frac{r^2}4))\hg(0)\le \hg(1)\le \hg(0)
$$
in $\wh B_0(p, r)$ for some positive constant $C_4$ independent of
$r$ and $T$ .  Hence we have $\wh B_1(p, \rho)\subset \wh
B_0(p,\frac r2)$ where
$\rho=C_4^\frac12\exp(\frac{1}{2}\exp(\frac{r^2}{4}))\frac r2$.

Given $\rho>0$, let $r$ be such that $\rho=C_4^\frac12\exp(\frac12
\exp(\frac {r^2}4))\frac r2$, rescale $\hg$ back to $g$, and
conclude from (\ref{fastdecayR-e2}) that
the lemma is true.

\end{proof}

Now we are ready to prove the theorem.

\begin{proof}[Proof of Theorem \ref{fastdecayR}] To prove (i), let
$\tg$ as in (iii). Then $\tg(s)$ solves the normalized \KR flow
equation:
\begin{equation}\label{fastdecayR-e3}
   \frac{\p \ttg_{i\jbar}}{\p t}=-\ttR_{i\jbar}-\ttg_{i\jbar}
\end{equation}
on $M\times(-\infty,\infty)$. By Lemma \ref{fastdecayR-L1}, the
fact that $\tg$ has positive Ricci curvature, and the fact that
the eigenvalues of $\Ric$ are nondecreasing by Theorem 6.1 in  \cite{ChauTam07-2},
we can proceed exactly as in the proof of Theorem \ref{positiveR}(i) to
show that the injectivity radius of $\tg(s)$ at $p$ is uniformly
bounded below away from zero. Rescaling back to $g(t)$, we see that
(i) is true. Note that here we use the fact that if we take a
limit along a sequence $s_i$ as in the proof of Theorem
\ref{positiveR}(i), then in the limit metric, the scalar
curvature at the origin is constant in time.

Using Lemma \ref{fastdecayR-L1} and (i), the proof of Theorem 1.1
in \cite{ChauTam06-2} can be carried over to prove that $M$ is
biholomorphic to $\cn$. \footnote{In Theorem 1.1 of
\cite{ChauTam06-2},   there exists a long time solution to \KRF
such that the estimate in Lemma \ref{fastdecayR-L1} holds for all
$x\in M$.   However, only the precise statement of Lemma
\ref{fastdecayR-L1} was actually used in the proof of the
Theorem.}

Since $\tg(s)$ is decreasing, for any $\rho>0$ and $s_0\ge 0$, the
curvature together with its derivatives are uniformly bounded on
$B_{s_0}(p,\rho)$ for all $s\ge s_0$. Combining with (i),  one
can apply the proof of the local version \cite[Corollary 3.18]{CCGGIIKLLN} of compactness of solutions
 of Hamilton \cite{Hamilton95-1}, to conclude that for any $s_i\to \infty$ we
can find a subsequence $s_{i_k}$ such that $(M,\tg(s+s_{i_k}), p)$
converges to a complete eternal solution of \eqref{fastdecayR-e3} which is
in fact an expanding gradient K\"ahler-Ricci soliton with  positive Ricci
curvature and nonnegative \hb. The fact that the limit solution is an
expanding gradient K\"ahler-Ricci soliton can be proved as in
\cite[\S2]{ChauTam07-1}. The fact that the limiting manifold is biholomorphic to
$\cn$ follows from  \cite{ChauTam05,Bryant07}.
\end{proof}

\begin{rem} Theorem \ref{fastdecayR} was proved by the authors in
\cite{ChauTam07-1} under the stronger assumption that $ tR(p, t)$
is uniformly bounded for all $t\geq 0$ independently of $p$.  In
particular, it was proved there that $M$ is biholomorphic to
$\C^n$.  We point out that the injectivity radius bound in
Theorem \ref{fastdecayR} (i) allows a direct application of the
methods in \cite{ChauTam06-2}, and thus a more direct proof of
this result.\end{rem}

By the result of \cite{Cao92}, if $tR(p,t)$ is not bounded, then
$tR(p,t)\to\infty$ as $t\to\infty$. We consider this case in the
next theorem.
\begin{thm}\label{slowdecayR} Suppose that at some $p\in M$,
$\lim_{t\to\infty}tR(p,t)=\infty$ and $\inf_tR(p,t)=0$.  Suppose also that
the eigenvalues values of $\Ric(p,t)$ are
uniformly pinched, i.e.  $ R_\ijb \ge cRg_\ijb$ at $(p,t)$  for
all $t$ for some $c>0$. Then there
exists $t_i\to\infty$ and   $a_i>0$ with $\lim_{i\to\infty}a_i=0$
such that the following are true.
\begin{itemize}
    \item [(i)]  The injectivity radius of $g(t_i)$ at $p$
        satisfies $\inj_{t_i}(p)\ge ca_i^{-1}$ for some $c>0$.
        \item [(ii)] $M^n$ is homeomorphic to Euclidean space and is
            biholomorphic to a pseduo-convex domain in $\cn$.
\item [(iii)] $(M^n,a_ig(t_i), p)$ converge to a steady
    gradient K\"aher Ricci soliton  with positive Ricci
    curvature which is biholomorphic to $\cn$.
\end{itemize}
\end{thm}
\begin{proof} To prove (i),
since $\inf_tR(p, t)= 0$ and $Ric>0$ for all $t$,   we can find
$T_i\to\infty$ such that $a_i=R(p,3T_i)=\inf_{t\in
[0,3T_i]}R(p,t)$. Then $a_i>0$ and $\lim_{i\to\infty}a_i=0$.
  Now the metrics
$g_i(t)=a_ig(T_i+\frac{t}{a_i})$ solve (\ref{krf-e1}) for
$-a_iT_i\le t<\infty$.   Moreover, by  the choice of $T_i$ we have
$R^{(i)}(p,t)\ge 1$ for $-a_iT_i\le t\le 2a_iT_i$ where $R^{(i)}$
is the scalar curvature of $g_i$. On the other hand, by Lemma 1.1,
 there
 is a constant $C_1$ independent of $i$ such that  $R^{(i)}(p,t)\le
C_1$  for all $0\le t\le 2a_iT_i$. Hence there is $0\le s_i\le
\frac14a_iT_i$ such that $|\frac{\p}{\p t}R^{(i)}(p,s_i)|\le
8C_1/(a_iT_i)$. By the assumption that $tR(p,t)\to\infty$ as
$t\to\infty$, we conclude that $\lim_{i\to\infty}a_iT_i=\infty$
and
\begin{equation}\label{slowdecayR-e1}
\lim_{i\to\infty}\frac{\p}{\p t}R^{(i)}(p,s_i)=0.
\end{equation}

We now consider the sequence of metrics $h_i(t)=g_i(s_i+t)$  which
solve  (\ref{krf-e1}) for $-a_iT_i-s_i\le t<\infty$.  By Lemma
\ref{harnack-L1} and the definition of $a_i$, there is a constand
$C_2$ independent of $i$ such that
\begin{equation}\label{slowdecayR-e2}
\wh R^{(i)}(x,t) \leq C_2
\end{equation}
 for all $x\in \wh B^{(i)}_0(p,\sqrt{a_iT_i})$ and
 $0\le t\le \frac12a_iT_i$, where $\wh R^{(i)}(x,t)$ is
the scalar curvature of $h_i(t)$ and $\wh B_t^{(i)}$ is the
geodesic ball with respect to $h^{(i)}$.

 Since the eigenvalues of the  Ricci curvature of $g(t)$ at $p$
 are uniformly pinched,   the Ricci curvature of $h^i(p, 0)$ is uniformly
 bounded
from below. Noting that $a_iT_i\to\infty$ as $i\to\infty$, by
\eqref{slowdecayR-e1} and \eqref{slowdecayR-e2},  we can proceed
as in the proof of Theorem \ref{positiveR}(i) to show that the
injectivity radius of $h^i(0)$ is uniformly bounded from below.
Rescaling back to $g(t)$, using the definition of $a_i$ we see
that (i) is true.

Part (ii)  follows from part (i) as in the proof of Theorem
\ref{positiveR}.

 By the above facts, in particular that (\ref{slowdecayR-e1})
  holds on a sequence of balls with radii going to $\infty$,
  one can proceed as in the proof of  Theorem \ref{fastdecayR} (iii) to conclude that $(M, h^i(0), p)$ converges a complete steady gradient K\"ahler-Ricci soliton with  positive Ricci
curvature and nonnegative \hb.  Note that the gradient of the
function defining the gradient K\"ahler-Ricci soliton has a unique
zero by \eqref{slowdecayR-e1} and the fact that the Ricci
curvature is positive. Hence the gradient K\"ahler-Ricci soliton
is biholomorphic to $\cn$ by the results in
\cite{ChauTam05,Bryant07}.

 \end{proof}

 The methods of proof in the theorems above allow us to study
 the holomorphic isometries of $(M,g(0))$.

 \begin{thm}\label{fixedpoint-t1}   Let  $(M, g(t))$  be as in the
 {\sc Basic Assumption 1}. Suppose in addition that
 $tR(p,t)\to\infty$ and the eigenvalues of the  $Rc(p, t)$ are uniformly pinched. Let $\Gamma$ be a finite subset of the
 holomorphic isometry group of $(M,g(0))$. Then there exists $q\in
 M$ such that $q$ is fixed by any $\gamma \in \Gamma$.
\end{thm}
\begin{proof} By Theorems \ref{positiveR} and \ref{slowdecayR},
we can find $t_i\to\infty$ and a bounded sequence $a_i>0$ such
that $(M,a_ig(t_i),p)$ converge to a steady gradient
K\"ahler-Ricci soltion with positive Ricci curvature. By the proof
of Theorem \ref{positiveR}, we conclude that there exists $a>0$
such that $\nabla_iR^{(i)}$ has a unique zero at $q_i\in
 B_i(p,a)$ which lies inside $B_i(p,\frac a3)$. Here $B_i$ is the
geodesic ball, $\nabla_i$ is the covariant derivative
 and $R^{(i)}$ is the scalar curvature with respect to
 $a_ig(t_i)$.

 We claim that for any holomorphic isometry of $(M,g(0))$, there
 is $i_0$ such that $\gamma(q_i)=q_i$ for all $i\ge i_0$. It is
 easy to see that the theorem follows from this claim as $\Gamma$
 is finite.

By the uniqueness of \KRF \cite{Chen-Zhu06} (see also
\cite{Fan07}), $\gamma$ is also an isometry of $g(t)$ and hence of
$a_ig(t_i)$ for all $i$. So $\gamma(q_i)$ is also a zero of
$\nabla_iR^{(i)}$.
 Since  $q_i$ is the unique zero of $\nabla_iR^{(i)}$ in
 $B_i(p,a)$ it is easy
 to see that  the claim is true if   $\gamma$ fixes $p$.

  Suppose $\gamma(p) \neq p$. Since for any $\e>0$,
  $B_{t_i}(p,\e)$ exhaust $M$, see \cite{ChauTam07-2}, and since
  $a_i$ is uniformly bounded, we conclude that there exists $i_0$
  such that if $i\ge i_0$, then $\gamma(p)\in B_i(p,\frac a3)$. Hence
   $\gamma(q_i)\in B_i(p,a)$. Since  $q_i$ is the unique zero of $\nabla_iR^{(i)}$ in
 $B_i(p,a)$, we have $\gamma(q_i)=q_i$ for $i\ge i_0$.
 \end{proof}

 In case $tR(p,t)$ is  uniformly bounded in $t$ or if $g(t)$ is an
 eternal solution, then we have the following much stronger
 result.

\begin{thm}\label{fixedpoint-t2}   Let  $(M, g(t))$  be as in the
 {\sc Basic Assumption 1}.  Suppose in addition that either (i)  for some $p\in
 M$, $tR(p,t)$ is uniformly bounded in $t$,  or (ii)
 $g(t)$ is an eternal solution.  Then there exists $q\in
 M$ such that $q$ is fixed by any holomorphic isometry of
 $(M,g(0))$.
\end{thm}
\begin{proof} First consider the case that $g(t)$ is an eternal
solution. Then by the monotonicity  of the eigenvalues of the
Ricci curvature in $t$ \cite{ChauTam07-2} (Theorem 6.1), the Ricci
curvature  at $p$ is uniformly bounded below away from zero for
all $t$. On the other hand, by the LYH type differential
inequality by Cao \cite{Cao97}, we know that $\frac{d}{dt}R(p, t)
\geq0$. Hence as in \cite{ChauTam07-1}, for any $t_i\to\infty$,
$(M,g(t_i),p)$ will subconverge to a gradient K\"ahler-Ricci
soliton with positive Ricci curvature and nonnegative \hb. Hence
as in the proof of Theorem \ref{fixedpoint-t1}, we can prove that
there exist  $T>0$ and some $a>0$ such that for $t\geq T$ there
exists $q(t) \in B_t(p, a)$ which is the unique point satisfying
$\nabla_t R (q(t), t) =0$ in $B_t(p, 3a)$.

Also as in the proof of Theorem \ref{fixedpoint-t1}, we can prove
that for every holomorphic isometry $\gamma$,  there exists
$T_{\gamma}$ such that $\gamma (q(t)) =q(t)$ for all $t\geq
T_{\gamma}$.

    We want to prove that  every isometry $\gamma$ fixes $q(T)$.
Suppose $\gamma(q(T))\neq q(T)$. Since $d_T(p,q(T))<a$
 by uniqueness of $q(T)$, we must have
$d_T(\gamma(q(T)), q(T))\geq 2a$.  On the other hand, we have
    $$d_{T_{\gamma}}(\gamma(q(T_{\gamma})), q(T_{\gamma}))=0$$
    and $T<T_\gamma$
     by the choice of $T_\gamma$.
  Now by uniqueness of $q(t)$ and the fact that $g(t)$ is a solution of the \KRF equation,
   it is not hard to see that $q(t)$ describes a continuous path on $M$
   with respect to $t$, and thus the same is true for $\gamma(q(t))$.
   It follows then, that for some $T<T_0<T_{\gamma}$ we must have
   $d_{T_0} (\gamma(q(T_0)), q(T_0))=a$.
   But this would violate the fact that
   $q(T_0)$ is the unique zero of $\nabla_{T_0}R$ in $B_{T_0}(p,
   3a)$. This proves the case when $g(t)$ is an eternal solution.

   The case that $t(R(p,t)$ is uniformly bounded in $t$ can be
   proved similarly using Theorem \ref{fastdecayR}.

\end{proof}

By the results of Shi \cite{Shi97} (see also \cite{ChauTam07-1}), if $(M, g)$ has bounded non-negative holomorphic bisectional curvature with average quadratic curvature decay and positive Ricci somewhere on $M$, then there exists a long time solution to \KRF such that $tR(p, t)$ is uniformly bounded independently of $p$.  We thus have the following immediate Corollary of Theorem \ref{fastdecayR} and \ref{fixedpoint-t2}:

\begin{cor}\label{quaddecaypositivericci}
Let $(M, g)$ be a complete non-compact \K manifold with non-negative holomorphic bisectional curvature.  Suppose $k(x, r) \leq C/(1+ r^2)$ for some $C$ and all $r$ and $x$, and that $Rc(p)>0$ for some $p\in M$.   Then $M$ is biholomorphic to $\C^n$, and there exists $q\in M$ such that $q$ is fixed by any holomorphic isometry of $(M, g)$.
\end{cor}

\section{Fiber bundle structures}

In this section we study structure of complete noncompact K\"ahler
manifold with nonnegative \hb\ using the \KR flow.  To motivate
our discussion, let us consider an example. Take a complete
noncompact K\"ahler metric $h$ on $\cn$ with positive \hb\  which
is $U(n)$ invariant, see \cite{Cao94,Cao97,WuZheng-pp}. Let
$\tM=\C^l\times \cn$ ($l\ge1$ with metric $\wt g=g_e\times h$
where $g_e$ is the standard metric on $\C^l$. Let $k$ be a nonzero
real number. Define $\gamma:\tM\to\tM$ as
$\gamma(\zeta,z)=(\zeta_1+k,\zeta_2,\dots,\zeta_l, e^{ 2\pi
k\ii}z)$ where $\zeta\in \C^l$ and $z\in \cn$. Then $\gamma$ is a
holomorphic isometry of $(\tM,\wt g)$ and $\gamma^m$ has no fixed
point unless $m=0$. Let $\Gamma$ be the cyclic group generated by
$\gamma$.  Then $(M,g)=(\tM,\wt g)/\Gamma$ is a complete noncompact
manifold with nonnegative \hb, which is a vector bundle over
$\C^l/\Gamma'$ where $\Gamma'$ is the group generated by $\gamma'$
which sends $(\zeta_1,\zeta_2,\dots,\zeta_l)$ to
$(\zeta_1+k,\zeta_2,\dots,\zeta_l)$.

Now let  $(M, g(t))$ be long time solution of the \KRF \eqref{krf-e1}
on complete noncompact K\"ahler manifold with bounded nonnegative
\hb\ yet without any assumption on the positivity of Ricci
curvature. If $(\tM, \tg(t))$ is the universal cover of $M$ with
corresponding lift $\tg(t)$, then $\tg(t)$ is also such a solution
to \KRF.  As mentioned in the introduction, $(\tM, \tg(t))$ splits
for each $t>0$ as $(\C^k \times N, g_e \times h(t))$ where $g_e$
is the standard Euclidean metric on $\C^n$ and $h(t)$ solves
(\ref{krf-e0}) and has positive Ricci curvature for $t>0$.  We
will use our results from \S2, for the factor $(N, h(t))$,
together with the arguments in \cite[\S 5]{NiTam03} to provide a
fiber bundle description of $M$.  In Corollary 3.1 we treat the
particular case of an initial complete non-compact \K manifold
having bounded non-negative \hb\ which decays quadratically in the
average sense.

We begin with the following

{\sc Basic assumption 2:} Let $M^n$ be a noncompact complex
manifold and let $(M^n,g(t))$  be a complete solution of the \KRF
\begin{equation}
   \frac{\p g_{i\jbar}}{\p t}=-R_{i\jbar}
\end{equation}
on $M\times[0,\infty)$ with uniformly bounded nonnegative
holomorphic bisectional curvature in spacetime, so that one of the
following is true:

\begin{itemize}
    \item [Case 1:] For some $p\in M$ (and hence for all
        point), $tR(p,t)$ is uniformly bounded in $t$.
    \item [Case 2:] For some point $p\in M$,
        $\lim_{t\to\infty}tR(p,t)=\infty$ and the positive
        eigenvalues of $\Ric(t)$ at $p$ are uniformly pinched
        for all $t$.
\end{itemize}

  \begin{thm}\label{fiberbundle}  With the {\sc Basic Assumption 2}, there is $k\ge0$
  such that for all $t$, $(M,g(t))$ is a holomorphic Riemannian
  fiber bundle over $\C^k/\Gamma$ where $\Gamma$ is a discrete
  subgroup of holomorphic isometries of $\C^k$, with fiber being
  $\C^{n-k}$ in Case 1 and a pseudoconvex domain
  in $\C^{n-k}$ in Case 2.  If in Case 2, $g(t)$ is in fact an eternal solution, then the fiber is also $\C^{n-k}$.
  \end{thm}
\begin{proof} Let $\tM$ be the universal cover of $M$.
By the splitting result of Cao \cite{Cao04}  we have
$\tM=\C^k\times \tN$ and $\ttg(t)=g_\e(t)\times \tth(t)$ where
$g_\e(t)$ is flat and $\tth(t)$ satisfies the {\sc Basic
Assumption 1} of \S 1 on $\tN$.  In particular, in Case 1 $\tN$ is
biholomorphic to $\C^{n-k}$ by Theorem \ref{fastdecayR}, and in
Case 2 $\tN$ is homeomorphic to the Euclidean space and
biholomorphic to a pseudoconvex domain in $\C^{n-k}$ by Theorem
\ref{positiveR} and \ref{slowdecayR}. Moreover, if in Case 2,
$g(t)$ and hence $\tth(t)$ is in fact an eternal solution, then by
\cite{ChauTam07-1}  $\tN$ is biholomorphic to $\C^n$.

Let $F\in \Gamma$, the first fundamental group of $M$ with respect
to the metric $g(0)$. By the uniqueness of \KRF (see
\cite{Chen-Zhu06,Fan07}), $F$ is an holomorphic isometry with
respect to $g(t)$ for all $t$. As a mapping on $\tM$, $F(\tx,
\ty)=(f_1(\tx, \ty),f_2(\tx, \ty))$ where $f_1:\tM\to \C^k$,
$f_2:\tM\to \tN$, $\tx\in \C^k$ and $\ty\in \tN$.

The following argument is basically as in \cite[\S 5]{NiTam03},
using the fact that $\tN$ is biholomorphic either to $\C^{n-k}$ or
a pseudoconvex domain in  $\C^{n-k}$ which is homeomorphic to
$R^{2(n-k)}$.

 We first prove $f_2$ is independent of $\tx$.
Suppose   there is $\tx_1\neq \tx_2$ and $\ty$ such that
$f_2(\tx_1,\ty)\neq f_2(\tx_2,\ty)$. Let $\alpha$ be a line
passing through $\tx_1$  and $\tx_2$ in $\C^k$. Then
$\alpha(s)=(\gamma(s),\ty)$ is a line in $\tM$. That is to say,
$\alpha$ is a geodesic which is minimizing between any two points
in $\alpha$.  Hence $F\circ \alpha(s)$ is a line of $\tM$ too. So
its projection onto $\tN$ is minimizing between any two end points
as a curve in $\tN$. But the projection is $f_2\circ \alpha$,
which is not a constant curve because $f_2(\tx_1,\ty)\neq
f_2(\tx_2,\ty)$. Hence it must be a line in $\tN$. But this
implies a factor $\C$ is splitted from $\tN$ by
\cite{CheegerGromoll71}. This is impossible because $\tN$ has
positive Ricci curvature. Thus the claim is established.

Next we want to prove that $f_1$ is independent of $\ty$. First
note that $F^{-1}=(h_1,h_2)$ with $h_2$ independent of $\tx$.
Hence $f_2\circ h_2$ is identity and $h_2=f_2^{-1}$. On the other
hand, since $F$ is an isometry, $f_2$ cannot increase length. For
if $\gamma(s)$ is a curve in $\tN$, then the length for the curve
$\alpha(s)=(\tx,\gamma(s))$ and the length of $F\circ
\alpha(s)=(f_1(\tx,\gamma(s)), f_2(\gamma(s))$ are equal. But the
length of $\alpha$ is the same as the length of $\gamma$. So the
length $f_2(\gamma(s))$ cannot be larger than the length of
$\gamma$. But this is also true for $f_2^{-1}$. Hence $f_2$ is an
isometry. If this is the case, then $f_1(\tx,\gamma(s))$ must be a
point. Hence $f_1$ is independent to $\ty$.

Let $\Lambda_1 =\{f_1| (f_1, f_2)\in \Gamma \text{ for some
$f_2$}\}.$ Then $\Lambda_1$ is a subgroup of the holomorphic
isometries of $(\C^k,g_\e(t))$ for all $t$. We claim that
$\Lambda_1$ is fixed point free. Suppose there is
$F'=(f_1',f_2')\in \Gamma$  such that $f_1'(\tx_0)=\tx_0$ for some
$\tx_0$.   Let $$G=\{f_2\in |\ F=(f_1, f_2)\in \Gamma \text{\ for
some $f_1$ with $f_1(\tx_0)=\tx_0$}\}.
$$
Then $G$ is a subgroup of the group of the holomorphic isometries of $\tN$.
$G$ acts on $\tN$ freely because $\Gamma$ acts on $\tM$ freely.
Now $\tth$ descends to $\tN/G$ while still satisfying the
{\sc Basic Assumption 1} in \S 1, and thus must be simply connected
by Theorems \ref{positiveR},\ref{fastdecayR} and \ref{slowdecayR} which leads to a contradiction if $G$ is not trivial.  Thus $f_1'$ must be trivial and $\Lambda_1$
is fixed point free.

Let $\beta$ be the  homomorphism from  $\Gamma$ to $\Lambda_1$
defined by $\beta(F)=f_1$, for $F=(f_1,f_2)$. Suppose
$F=(f_1',f_2')\in \Gamma$ and $f_1'$ is the identity map. Then by
similar argument, we can prove that the group:

$$
\{f_2|\  (f_1',f_2)\in \Gamma\}
$$
is trivial. Hence $\beta$ is an isomorphism. Then $\Lambda_1$ acts
on $\tM$ isometrically holomorphically and freely by
$f(\tx,\ty)=(f(\tx),\rho_2\circ\beta^{-1}(f)(\ty))$ where $\rho_2$
is the homomorphism given by mapping $F=(f_1,f_2)\in \Gamma$ to
$f_2$.

 Hence
as shown in    \cite[\S5]{NiTam03},  $M$ is a Riemannian and holomorphic
fiber bundle over $\C^k/\Lambda_1$ with fiber $\tN$.  By the comments in the first paragraph of the proof, we see that this completes the proof of the Theorem.
\end{proof}

  \begin{thm}\label{linebundle} If $n-k=1$ in the previous Theorem,  then $(M,g(t))$ is a holomorphic Riemannian
  line bundle over $\C^k/\Gamma$
  \end{thm}
 \begin{proof} Let $\Gamma$ and $\tN$ as in the proof of previous
 theorem. Then $\tN$ is bihlomorphic to $\C$ and has a global holomorphic coordinate $z$.
 Thus for any $F\in
 \Gamma$ with $F=(f_1,f_2)$, $f_2(z)=az+b$ for some $a\neq 0$ and $b\in \C$.
  Now by Theorem \ref{fixedpoint-t1}, any $f_2$
 has a fixed point.  It follows that any nontrivial $f_2$ with $F=(f_1,f_2)\in \Gamma$,
  has a
 unique fixed point.  By Theorem \ref{fixedpoint-t1} again, we conclude that
 there is $z_0\in \tN=\C$ such that $f_2(z_0)=z_0$ for any holomorphic isometry
 $f_2$ with $ F=(f_1,f_2)\in \Gamma$.
  We see that we can choose a global holomorphic coordinate $\tilde{z}$ on $\tN=\C$  such that in this coordinate,
  every such $f_2$ has the form $f_2(\tilde{z})=a \tilde{z}$ for some complex constant $a$.
  Thus in this case, the fiber bundle structure constructed in
  the proof of Theorem \ref{fiberbundle}
  is in fact a line bundle structure.  This completes the proof of the Theorem.
 \end{proof}

By the remarks preceding Corollary \ref{quaddecaypositivericci}, we may have the following immediate corollary to Theorems \ref{fiberbundle} and \ref{linebundle}:
\begin{cor}\label{quaddecay}
Let $(M, g)$ be a complete non-compact \K manifold with non-negative holomorphic bisectional curvature.  Suppose $k(x, r) \leq C/(1+ r^2)$ for some $C$ and all $r$ and $x$.    Then there is $k\ge0$
  such that $(M,g)$ is a holomorphic Riemannian
  fiber bundle over $\C^k/\Gamma$ where $\Gamma$ is a discrete
  subgroup of holomorphic isometries of $\C^k$, with fiber being
  $\C^{n-k}$.

  Moreover, if $n-k =1$ then the above fiber bundle is in fact a line bundle.
  \end{cor}

\section{Gradient K\"ahler-Ricci Solitons}
In this section we refine the results in the previous section in the case of a steady gradient \KR  soliton with bounded and
 nonnegative holomorphic bisectional curvature. We remark that in the case
 of an expanding gradient \KR soliton satisfying the same curvature
 conditions,
  we have $M$ biholomorphic to $\C^n$ by
 \cite{ChauTam05}, see also \cite{Bryant07}.

   First let us fix
 some notations. Let $( M^n, g)$ be a complete steady gradient
 \KR Ricci soliton with bounded and
 nonnegative holomorphic bisectional curvature. Namely, there is a real
 valued function $f$ such that
  \begin{equation}\label{1.e4}
\begin{split}
f_{i\jbar }&=   R_{i\jbar} \\
f_{ij}&=0.
\end{split}
\end{equation}
 Let $(\tM,\ttg)$
 be the universal cover of $(M^n,g)$. By \cite{Cao04}, $\tM=\C^k\times \tN$
 and $\ttg=g_e\times \tth$,   where $g_e$ is the standard metric
 and $(\tN,\tth)$ is a steady gradient \KR soliton with bounded
 nonnegative holomorphic bisectional curvature and with positive
 Ricci curvature. Hence
there exists some smooth real valued function $f$ satisfying
\eqref{1.e4} on $\tN$. Now let us also assume that the scalar
curvature of $M$ attains its maximum at some point. Hence the same
is true for $\tN$. Then by \cite{CaoHamilton2000} $\nabla f$(the real gradient of $f$ relative to $\tilde{h}$) has a
unique zero at some $\tilde p$. Let
$\lambda_1\ge\lambda_2\ge\dots\ge\lambda_l>0$ be the eigenvalues of
the Ricci curvature at $\tilde p$, where $l=n-k$. Recall that
$\tN$ is biholomorpic to $\C^l$ by \cite{ChauTam05},
\cite{Bryant07}. Let $\Gamma$ be the fundamental group of $M$.
Then by the proof of Theorem \ref{fiberbundle}, each $F$ in
$\Gamma$ is of the form $(\gamma_1,\gamma_2)$ where $\gamma_1$ is
an holomorphic isometry of $\C^k$ and $\gamma_2$ is a holomorphic
isometry of $(\tN,\tth)$. Let
$$
\Lambda_1=\{\gamma_1|\ (\gamma_1,\gamma_2)\in\Gamma\
\text{\rm for some $\gamma_2$}\},
$$
and
$$
\Lambda_2=\{\gamma_2|\ (\gamma_1,\gamma_2)\in\Gamma\
\text{\rm for some $\gamma_1$}\}.
$$

We want to describe the holomorphic isometries in $\Lambda_2$. In
fact, we want to describe any holomorphic isometry of $\tN$ in general. It
was shown by Bryant \cite{Bryant07} that there is a global
holomorphic coordinates $z^i$ of $\tN$, ranging over $\C^l$, which linearize the
holomorphic vector field  $Z=\frac12\lf(\nabla f-\ii J(\nabla
f)\ri)$. That is to say,
$$
Z=\sum_{i=1}^l\lambda_i\frac{\p}{\p z^i}.
$$
Such coordinates are called Poincar\`e coordinates.
\begin{rem} We remark that in case of expanding gradient
K\"ahler-Ricci soliton with nonnegative Ricci curvature,
Poincar\'e coordinates also exist. This is because in  this case,
the potential function $f$ has a unique critical point and
$J(\nabla f)$ is also a Killing vector field.
\end{rem}

   Now let $\phi:(\tN,\tth) \to (\tN, \tth)$ be a holomorphic
   isometry.  Thus $\phi^*(\tth)=\tth$.
\begin{lem}\label{s2l0} The following are true:

\begin{enumerate}
  \item $\phi^* (f)=f$.
 \item $\phi_*(Z) = Z$.
 \end{enumerate}
 \end{lem}
 \begin{proof} Let $f_1=\phi^*(f)$.  Since
$\phi^*(\tth) =\tth$ we   have $\nabla^2 (f_1)= \phi^*(\nabla^2
f)$ and thus $(f_1)_{\ijb}=R_\ijb=f_\ijb$. Hence $\nabla f_1$ has
a unique zero, which must be at $\tilde p$. Hence $\phi(\tilde
p)=\tilde p$.  So $f_1(\tilde p)=f(\tilde p)$. On the other hand,
since $\nabla f(\tilde p)=\nabla f_1(\tilde p)=0$ and
$\nabla^2(f-f_1)=0$, we conclude that (1) is true.
  (2) follows from the definition of $Z$,  the fact that
  $\phi^*(\tth)=\tth$ and
from (1).  \end{proof}

\begin{lem}\label{s2l1} Let $\phi$ be a holomorphic
isometry of $(\tN,\tth)$. Let $\phi(z)=w$ in the Poincar\'e
coordinates. Then $\phi$ is an upper triangular polynomial map in the sense that
$$
w=A(z)+G(z).
$$
where $A(z)$ is a linear map and $G=(G_1,\dots,G_l)$ with
$G_i=G_i(z^{i+1},\dots,z^l)$. If in addition, for all $j$
$$
\sum_{i=1}^l m_i\lambda_i\neq \lambda_j
$$
for all nonnegative integers $m_i$  such that
$\sum_{i=1}^lm_i\ge2$, then $\phi$ is a linear map.
 \end{lem}
 \begin{proof} This is basically a result in
 \cite{Bryant07}. For the sake of completeness, we sketch
 the proof here. Since $\phi_*(Z)=Z$ and we are using
 Poincar\'e coordinates in the domain and target, for each
 $i$, we have
  \begin{equation}\label{s2e1}
  \lambda_i w^i=\sum_{j=1}^l\lambda_j z^j\frac{dw^i}{dz^j}
  \end{equation}
  everywhere.  On the other hand for each $i$, as a function
  of $z$ we know that $w^i$ has a global power series
  representation
  \begin{equation}\label{s2e2}
w^i=\sum_{|\alpha|=1}^{|\alpha|=\infty}c_{\alpha}(i)
z^{\alpha} \end{equation} where $\alpha$ is a multi-index.
The fact that there is no constant term in (\ref{s2e2}) follows
from the the fact that by their definition, both $z$ and $w$ must
vanish at  $\tilde p \in \tN$.   Now by substituting the power
series in (\ref{s2e2}) and its derivative with respect to $z$ for
both sides of (\ref{s2e1}), equating coefficients and using the
fact that the $\lambda_i$'s are all positive, the result follows.

 \end{proof}

In    \cite{Bo08}, Yang proves the following: Let $M$ be a
noncompact K\"ahler-Ricci soliton with nonnegative Ricci
curvature. Suppose its scalar curvature attains a positive maximum
at a compact complex submanifold $K$ with codimension 1 and the
Ricci curvature is positive away from $K$. Then $M$ is a
holomorphic line bundle over $K$.

In case the \hb\ is bounded and nonnegative, from the lemma, the
proof of Theorem \ref{fiberbundle} and the discussion before Lemma
\ref{s2l0},   we have the following:
\begin{thm}\label{steady-s3t2}  Let $(M, g)$ be a steady gradient \KR soliton with bounded non-negative holomorphic bisectional curvature such that the scalar curvature is maximal at some $p\in M$. Then assuming the above notation,  there exists global complex coordinates on $\tN$
such that $M$ is a Riemannian and holomorphic fiber bundle over
$\C^k/\Lambda_1$ with fiber $C^{n-k}$ such that each $\gamma_2\in \Lambda_2$ is of
upper triangular form. If in addition, for all positve $\lambda_j$ we have
$$
\sum_{i=1}^l m_i\lambda_i\neq \lambda_j
$$
for all nonnegative integers $m_i$  such that
$\sum_{i=1}^lm_i\ge2$. Then $M$ is a holomorphic vector bundle
over $\C^k/\Lambda_1$.
\end{thm}

We end this section by proving a general volume growth estimate
for expanding gradient   solitons. The motivation is from the work
of Ni \cite{Ni05-1} where he proved that a complete noncompact
K\"haler manifold with bounded and nonnegative \hb\ and with
maximum volume growth must have quadratic curvature decay as in
(1) in the Introduction.

 \begin{prop}\label{volume-s3t1} Suppose $(M^n,g)$ is a complete Riemannian
 manifold such that $g$ is an expanding gradient
 Ricci soliton with nonnegative scalar curvature. Then for any $o\in M$, there is a
 constant $C>0$ such that $V_o(r)\ge Cr^n$ for all $r\ge 1$, where $V_o(r)$ is the volume
 of geodesic ball of radius $r$ with center at $o$.
 \end{prop}
 \begin{proof} Let $f$ be such that $f_{ij}=R_{ij}+g_{ij}$. By the
 proof     in \cite[\S20]{Hamilton95-2} (see also \cite{CCGGIIKLLN}),
 $\mathcal{R}+|\nabla f|^2-2f$ is a constant, where $R$ is the scalar curature. Modifying $f$
by adding a constant to $f$, we may assume that
$\mathcal{R}+|\nabla f|^2-2f=0$. Since $\mathcal{R}\ge0$, we have
$f\ge0$ and $|\nabla f|^2\le 2f.$  Fix   $o\in M$. For any $x$ let
$\gamma(s)$ be a minimal geodesic from $o$ to $x$ parametrized by
arc length and let $r=d(o,x)$. Then for any $0\le \tau\le r$
\begin{equation}\label{vol-soliton-e3}
   \begin{split}
 \frac{d}{d\tau}f(\gamma(\tau))&\le |\nabla f|\\
 &\le (2f)^\frac12.
\end{split}
\end{equation}
Hence $f^\frac12(x)\le \frac 1{\sqrt 2} r+f^\frac12(o)$. This
implies that $|\nabla f|(x) \le r+\lf(2f(o)\ri)^\frac12 $.

On the other hand, $\Delta f\ge n$, where $n$ is the dimension of
$M$. For any $R$
\begin{equation}\label{vol-soliton-e4}
   \begin{split}
nV(R)&\le \int_{B(R)} \Delta f\\
&=\int_{\p B(R)}\frac{\p f}{\p \nu}\\
&\le \int_{\p B(R)}|\nabla f|\\
&\le \lf(R+\lf(2f(o)\ri)^\frac12\ri)A(R)
\end{split}
\end{equation}
where $A(R)$ is the area of $\p B_o(R)$ and $V(R)=V_o(R)$. Hence
$$V(R)\ge
\lf(\frac{R+\lf(2f(o)\ri)^\frac12}{1+\lf(2f(o)\ri)^\frac12}\ri)^nV(1)
$$
for $R\ge1$.

 \end{proof}

\bibliographystyle{amsplain}

  \end{document}